\documentclass[12pt,fleqn,leqno]{amsart}
\usepackage{amsfonts,amssymb}
\usepackage{mathrsfs}
\usepackage{enumitem}

\usepackage[svgnames]{xcolor}
\usepackage[colorlinks,linkcolor=blue,citecolor=Green]{hyperref}
\usepackage{titlesec}

\titleformat{\section}[block]
 {\bfseries}
 {\thesection.}
 {\fontdimen2\font}
 {}

\newcommand{\periodafter}[1]{#1.}

\titleformat{\subsection}[runin]
 {\bfseries}
 {\thesubsection.}
 {\fontdimen2\font}
 {\periodafter}

\setlist{noitemsep}
\setenumerate{labelindent=\parindent,label=\upshape{(\alph*)}}

\newtheorem{theorem}{Theorem}[section]
\newtheorem{corollary}[theorem]{Corollary}
\newtheorem{proposition}[theorem]{Proposition}
\newtheorem{lemma}[theorem]{Lemma}

\theoremstyle{definition}
\newtheorem{question}{Question}

\DeclareMathOperator{\cut}{ct}
\DeclareMathOperator{\noncut}{nct}

\DeclareMathAlphabet{\mathpzc}{OT1}{pzc}{m}{it}

\DeclareMathOperator{\sel}{\mathpzc{V\mkern-5mu_{cs}}}

\renewcommand{\emptyset}{\varnothing}
\numberwithin{equation}{section}

\begin{document}

\author{Valentin Gutev}
  \address{Institute of Mathematics and Informatics, Bulgarian Academy of Sciences,
Acad. G. Bonchev Street, Block 8, 1113 Sofia, Bulgaria}

\email{\href{mailto:gutev@math.bas.bg}{gutev@math.bas.bg}}



\title{Butterfly Points and Hyperspace Selections}
\begin{abstract}
  If $f$ is a continuous selection for the Vietoris hyperspace
  $\mathscr{F}(X)$ of the nonempty closed subsets of a space $X$, then
  the point $f(X)\in X$ is not as arbitrary as it might seem at first
  glance. In this paper, we will characterise these points by local
  properties at them. Briefly, we will show that $p=f(X)$ is a strong
  butterfly point precisely when it has a countable clopen base in
  $\overline{U}$ for some open set $U\subset X\setminus\{p\}$ with
  $\overline{U}=U\cup\{p\}$.  Moreover, the same is valid when $X$ is
  totally disconnected at $p=f(X)$ and $p$ is only assumed to be a
  butterfly point. This gives the complete affirmative solution to a
  question raised previously by the author. Finally, when $p=f(X)$
  lacks the above local base-like property, we will show that
  $\mathscr{F}(X)$ has a continuous selection $h$ with the stronger
  property that $h(S)=p$ for every closed $S\subset X$ with $p\in S$.
\end{abstract}

\date{\today}
%\date{}
\maketitle

\section{Introduction}

All spaces in this paper are infinite Hausdorff topological
spaces. Let $\mathscr{F}(X)$ be the set of all nonempty closed subsets
of a space $X$. We endow $\mathscr{F}(X)$ with the \emph{Vietoris
  topology} $\tau_V$, and call it the \emph{Vietoris hyperspace} of
$X$. Let us recall that $\tau_V$ is generated by all collections of
the form
\[
\langle\mathscr{V}\rangle = \left\{S\in \mathscr{F}(X) : S\subset
  \bigcup \mathscr{V}\ \ \text{and}\ \ S\cap V\neq \emptyset,\
  \hbox{whenever}\ V\in \mathscr{V}\right\},
\]
where $\mathscr{V}$ runs over the finite families of open subsets of
$X$. A map $f:\mathscr{F}(X)\to X$ is a \emph{selection} for
$\mathscr{F}(X)$ if $f(S)\in S$ for every $S\in\mathscr{F}(X)$, and
$f$ is \emph{continuous} if it is continuous with respect to the
Vietoris topology on $\mathscr{F}(X)$. The set of all continuous
selections for $\mathscr{F}(X)$ will be denoted by
$\sel[\mathscr{F}(X)]$.\medskip

The following question was posed by the author in an
unpublished note.

\begin{question}
  \label{question-rem-obser-v4:3}
  Let $\omega_1$ be the first uncountable ordinal and
  $X=(\omega_1+1)\cup_{\omega_1=0}[0,1]$ be the adjunction space
  obtained by identifying $\omega_1$ and $0\in[0,1]$
  into a single point $p\in X$. Does there exist a continuous
  selection $f:\mathscr{F}(X)\to X$ with $f(X)=p$?
\end{question}

David Buhagiar has recently proposed, in a private communication to
the author, a negative solution to this question. However, his
arguments were heavily dependent on stationary sets in $\omega_1$ and
the pressing down lemma.\medskip

In this paper, we will give a purely topological description of the
points $p\in X$ with the property that ${p=f(X)}$ for some
$f\in \sel[\mathscr{F}(X)]$. To this end, let us recall that a point
$p\in X$ of a connected space $X$ is \emph{cut} if $X\setminus \{p\}$
is not connected or, equivalently, if $X\setminus\{p\}=U\cup V$ for
some subsets $U,V\subset X$ with
$\overline{U}\cap \overline{V}=\{p\}$. Extending this interpretation
to an arbitrary space $X$, a point $p\in X$ was said to be \emph{cut}
\cite{gutev-nogura:03a}, see also \cite{gutev:00e,gutev-nogura:00d},
if $X\setminus\{p\}=U\cup V$ and $\overline{U}\cap \overline{V}=\{p\}$
for some subsets $U,V\subset X$. Cut points were also introduced in
\cite{Dow2008}, where they were called \emph{tie-points}. A somewhat
related concept was introduced in \cite{gutev:05a}, where $p\in X$ was
called \emph{countably-approachable} if it is either isolated or has a
countable clopen base in $\overline{U}$ for some open set
$U\subset X\setminus\{p\}$ with $\overline{U}=U\cup\{p\}$. In these
terms, we consider the following two subsets of $X$:
\begin{equation}
  \label{eq:HSAP-v20:2}
  \begin{cases}
    X_\Theta=\left\{f(X): f\in \sel[\mathscr{F}(X)]\right\}, \\
    X_\Omega=\{p\in X: p\ \text{is countably-approachable}\}.
  \end{cases}
\end{equation}
Intuitively, $X_\Theta$ can be regarded as the $X$-`Orbit' with
respect to the `action' of $\sel[\mathscr{F}(X)]$ on the hyperspace
$\mathscr{F}(X)$.  The non-isolated countably-approachable points were
called \emph{$\omega$-approachable} in \cite{gutev:05a}, so $X_\Omega$
is also intuitive. \medskip

Each non-isolated countably-approachable point $p\in X$ is a cut point
of $X$, see the proof of \cite[Corollary 3.2]{gutev-nogura:03a}. Going
back to the adjunction space $X$ in
Question~\ref{question-rem-obser-v4:3}, it is evident that the point
$p\in X$ is cut, but it is not countably-approachable (i.e.\
$p\notin X_\Omega$). The reason for the negative answer to this
question (i.e.\ for the fact that $p\notin X_\Theta$) is now fully
explained by our first main result.

\begin{theorem}
  \label{theorem-HSAP-v7:1}
  Let $X$ be a space with $\sel[\mathscr{F}(X)]\neq\emptyset$. Then
  \begin{equation}
    \label{eq:HSAP-v20:4}
    X_\Omega=\{p\in X_\Theta: p\ \text{is either isolated or cut}\}.
  \end{equation}
\end{theorem}

Regarding the proper place of Theorem \ref{theorem-HSAP-v7:1}, let us
remark that the inclusion $X_\Omega\subset X_\Theta$ was actually
established in \cite[Lemma 4.2]{gutev:05a}, see also \cite[Lemma
2.3]{gutev-nogura:00b}. It is included in \eqref{eq:HSAP-v20:4} to
emphasise on the equality. The other inclusion is naturally related to
the so called butterfly points.  A point $p\in X$ is called
\emph{butterfly} (or, a \emph{b-point}) \cite{Sapirovskii1975} if
$\overline{F\setminus\{p\}}\cap \overline{G\setminus\{p\}}=\{p\}$ for
some closed sets $F,G\subset X$. Evidently, $p\in X$ is butterfly
precisely when it is a cut point of some closed subset of $X$. In
fact, in some sources, cut points (equivalently, tie-points) are often
called \emph{strong butterfly points}. Finally, let us agree that a
space $X$ is \emph{totally disconnected at $p\in X$} if the singleton
$\{p\}$ is an intersection of clopen sets.\medskip

A crucial role in the proof of Theorem \ref{theorem-HSAP-v7:1} will
play the following result. 

\begin{theorem}
  \label{theorem-HSAP-v7:2}
  A point $p\in X_\Theta$ is butterfly if and only if it is the limit
  of a sequence of points of $X\setminus\{p\}$. Moreover, if $X$ is
  totally disconnected at a butterfly point $p\in X_\Theta$, then $p$
  is also a cut point of $X$.
\end{theorem}

Evidently, each $p\in X$ which is the limit of a nontrivial sequence
of points of $X$ is also a butterfly point. Thus, the essential
contribution in the first part of Theorem \ref{theorem-HSAP-v7:2} is
that each butterfly point $p\in X_\Theta$ is the limit of a nontrivial
sequence of points $X$. However, butterfly points $p\in X_\Theta$ are
not necessarily cut (i.e.\ strong butterfly). For instance, the
endpoints $0,1\in [0,1]$ are noncut, they belong to $[0,1]_\Theta$ and
are both butterfly. In contrast, according to Theorem
\ref{theorem-HSAP-v7:1}, the second part of Theorem
\ref{theorem-HSAP-v7:2} implies the following consequence which
settles \cite[Problem 4.15]{gutev-2013springer} in the affirmative.

\begin{corollary}
  \label{corollary-HSAP-v7:1}
  If $X$ is totally disconnected at a butterfly point $p\in X_\Theta$,
  then $p\in X_\Omega$ or, equivalently, this point is
  countably-approachable.
\end{corollary}

Let us remark that \cite[Problem 4.15]{gutev-2013springer} was stated
for a space $X$ which has a clopen $\pi$-base. A family $\mathscr{P}$
of open subsets of $X$ is a \emph{$\pi$-base} (called also a
\emph{pseudobase}, Oxtoby \cite{oxtoby:60}) if every nonempty open
subset of $X$ contains some nonempty member of $\mathscr{P}$. In our
case, we also have that $\sel[\mathscr{F}(X)]\neq \emptyset$ and
according to \cite[Corollary 2.3]{gutev:05a}, such a space $X$ must be
totally disconnected. \medskip

Our second main result deals with the elements of the set
$X_\Theta\setminus X_\Omega$. To this end, let us recall that a point
$p\in X$ is called \emph{selection maximal} \cite{gutev-nogura:03a},
see also \cite{garcia-ferreira-gutev-nogura-sanchis-tomita:99,
  gutev-nogura:00b}, if there exists a continuous selection $f$ for
$\mathscr{F}(X)$ such that $f(S)=p$ for every $S\in \mathscr{F}(X)$
with $p\in S$. In this case, the selection $f$ is called
\emph{$p$-maximal}. Evidently, each selection maximal point of $X$
belongs to $X_\Theta$, and each point of $X$ which has a
countable clopen base belongs to $X_\Omega$. So, consider the sets:
\begin{equation}
  \label{eq:HSAP-v21:1}
  \begin{cases}
    X_\Theta^*=\left\{p\in X_\Theta: p\ \text{is selection
        maximal}\right\}, \\
    X_\Omega^*=\left\{p\in X_\Omega: p\ \text{has a countable clopen
        base}\right\}.
  \end{cases}
\end{equation}

In this paper, we will also prove the following theorem.

\begin{theorem}
  \label{theorem-HSAP-v7:3}
    Let $X$ be a space with $\sel[\mathscr{F}(X)]\neq\emptyset$. Then
    \begin{equation}
      \label{eq:HSAP-v21:3}
      X_\Theta\setminus X_\Omega\subset X_\Theta^*\quad
      \text{and}\quad X_\Theta^*\cap X_\Omega=X_\Omega^*.
    \end{equation}
\end{theorem}

Theorem \ref{theorem-HSAP-v7:3} is also partially known, the equality
$X_\Theta^*\cap X_\Omega=X_\Omega^*$ was established in \cite[Theorem
3.1 and Corollary 3.2]{gutev-nogura:03a}. It is included to emphasise
on the fact that $X_\Theta\setminus X_\Omega$ is not necessarily equal
to $X_\Theta^*$. According to Theorem \ref{theorem-HSAP-v7:1}, the set
$X_\Theta\setminus X_\Omega$ cannot contain a cut point of $X$. A
point $p\in X$ which is not cut will be called \emph{noncut}, see
Section \ref{sec:point-maxim-select}. Thus, the inclusion
$X_\Theta\setminus X_\Omega\subset X_\Theta^*$ in
\eqref{eq:HSAP-v21:3} actually states that each noncut point
$p\in X_\Theta$ is selection maximal. The crucial property to achieve
this result is that the connected component of each noncut point
$p\in X_\Theta$ has a clopen base (Theorem
\ref{theorem-Sequences-v13:1}). \medskip

The paper is organised as follows. Theorem \ref{theorem-HSAP-v7:2} is
proved in Section \ref{sec:butt-points-conv}. A condition for a point
$p\in X_\Theta$ to be countably-approachable is given in Lemma
\ref{lemma-Sequences-v3:1} of Section
\ref{sec:count-appr-points}. Based on this condition, the proof of
Theorem \ref{theorem-HSAP-v7:1} is accomplished in Sections
\ref{sec:appr-totally-disc} and \ref{sec:appr-nontr-comp}. The final
Section \ref{sec:point-maxim-select} contains the proof of Theorem
\ref{theorem-HSAP-v7:3}.

\section{Butterfly Points and Convergent Sequences} 
\label{sec:butt-points-conv}

A nonempty subset $S$ of a partially ordered set $(P,\leq)$ is
\emph{up-directed} if for every finite subset $T\subset S$ there
exists $s\in S$ with $t\leq s$ for every $t\in T$.  For a space $X$,
the set $\mathscr{F}(X)$ is partially ordered with respect to the
usual set-theoretic inclusion `$\subset$', and each up-directed family
in $\mathscr{F}(X)$ is $\tau_V$-convergent.
\begin{proposition}
  \label{proposition-G-property-v5:2}
  Each up-directed family $\mathscr{S}\subset \mathscr{F}(X)$ is
  $\tau_V$-convergent to $\overline{\bigcup\mathscr{S}}$.
\end{proposition}

\begin{proof}
  Let $\mathscr{V}$ be a finite family of open subsets of $X$ with
  $\overline{\bigcup\mathscr{S}}\in \langle\mathscr{V}\rangle$. Then
  $\bigcup \mathscr{T}\in \langle\mathscr{V}\rangle$ for some finite
  subfamily $\mathscr{T}\subset \mathscr{S}$. Since $\mathscr{S}$ is
  up-directed, $\bigcup\mathscr{T}\subset S$ for some
  $S\in \mathscr{S}$, and each $S\in \mathscr{S}$ with this property
  also belongs to $\langle\mathscr{V}\rangle$.
\end{proof}

Complementary to Proposition \ref{proposition-G-property-v5:2} is the
following further observation about $\tau_V$-convergence of usual
sequences in the hyperspace $\mathscr{F}(X)$. 

\begin{proposition}
  \label{proposition-G-property-v3:1}
  Let $U_n\subset X$, $n<\omega$, be a pairwise disjoint family of
  proper open sets. Then the sequence $S_n=X\setminus U_n$,
  $n<\omega$, is $\tau_V$-convergent to $X$.
\end{proposition}

\begin{proof}
  Take a finite open cover $\mathscr{V}$ of $X$ with
  $X\in \langle\mathscr{V}\rangle$. If $S_k\cap V_0=\emptyset$ for
  some $V_0\in \mathscr{V}$ and $k<\omega$, then $V_0\subset
  U_k$. Since $\{U_n: n<\omega\}$ is pairwise disjoint, this implies
  that $\emptyset\neq V_0\subset U_k\subset S_n$ for every $n\neq k$.
  Since $\mathscr{V}$ is finite, there exists $n_0<\omega$ such that
  $S_n\cap V\neq \emptyset$ for every $V\in \mathscr{V}$ and
  $n\geq n_0$. In other words, $S_n\in \langle\mathscr{V}\rangle$ for
  every $n\geq n_0$.
\end{proof}

In what follows, for a set $Z$, let
\begin{equation}
  \label{eq:HSAP-v21:4}
  \Sigma(Z)=\{S\subset Z: S\ \text{is nonempty and finite}\}.
\end{equation}
The following two general observations about local bases generating
nontrivial convergent sequences furnish the first part of the proof of
Theorem \ref{theorem-HSAP-v7:2}.

\begin{proposition}
  \label{proposition-Sequences-v1:1}
  Let $p=f(X)$ be a non-isolated point for some
  $f\in \sel[\mathscr{F}(X)]$, and $\mathscr{B}$ be a local base at
  $p$ such that $f((X\setminus B)\cup\{p\})\in X\setminus B$ for every
  $B\in \mathscr{B}$.  Then $X\setminus \{p\}$ contains a sequence
  convergent to $p$.
\end{proposition}

\begin{proof}
  For $B_0\in \mathscr{B}$ and
  $q=f\left((X\setminus B_0)\cup \{p\}\right)\neq p$, there are
  disjoint open sets $O_p,O_q\subset X$ with $q\in O_q$ and
  $p\in O_p\subset B_0$. Hence, by continuity of $f$, there is a
  finite family $\mathscr{V}$ of open sets of $X$ with
  $(X\setminus B_0)\cup \{p\}\in \langle\mathscr{V}\rangle$ and
  $f(\langle\mathscr{V}\rangle)\subset O_q$. Take $B_1\in \mathscr{B}$
  such that $B_1\subset O_p$ and
  $B_1\subset \bigcap\{V\in \mathscr{V}: p\in V\}$. If
  $S\in \Sigma(B_1)$, then
  $f\left((X\setminus B_0)\cup S\right)\in X\setminus B_0$ because
  $(X\setminus B_0)\cup S\in \langle\mathscr{V}\rangle$ and
  $S\subset O_p\subset X\setminus O_q$.  Since
  $\left\{(X\setminus B_0)\cup S: S\in \Sigma(B_1)\right\}$ is
  up-directed and $\bigcup\Sigma(B_1)=B_1$, by Proposition
  \ref{proposition-G-property-v5:2},
  $f\left((X\setminus B_0)\cup \overline{B_1}\right)\in X\setminus
  B_0$. Thus, by induction, there is a decreasing sequence
  $\{B_n\}\subset \mathscr{B}$ such that
  $f\left((X\setminus B_n)\cup \overline{B_{n+1}}\right)\in X\setminus
  B_n$ for every $n<\omega$. Then $B_n\setminus\overline{B_{n+1}}$,
  $n<\omega$, is a pairwise disjoint family of proper open subsets of
  $X$. Hence, by Proposition \ref{proposition-G-property-v3:1}, the
  sequence $T_n=(X\setminus B_n)\cup \overline{B_{n+1}}$, $n<\omega$,
  is $\tau_V$-convergent to $X$. So, $p=f(X)=\lim_{n\to\infty}f(T_n)$
  and $f(T_n)\notin B_n\ni p$, $n< \omega$.
\end{proof}

\begin{proposition}
  \label{proposition-Sequences-v1:2}
  Let $p=f(X)$ be a butterfly point for some
  $f\in \sel[\mathscr{F}(X)]$, and $\mathscr{B}$ be a local base at
  $p$ such that $f((X\setminus B)\cup\{p\})= p$ for every
  $B\in \mathscr{B}$.  Then $X\setminus \{p\}$ contains a sequence
  convergent to $p$.
\end{proposition}

\begin{proof}
  By definition, there are closed sets $F,G\subset X$ with
  $\overline{F\setminus\{p\}}\cap \overline{G\setminus\{p\}}=\{p\}$.
  Set $U=F\setminus\{p\}$ and $V=G\setminus\{p\}$, and take
  $B_0\in \mathscr{B}$. Since
  $f\left((X\setminus B_0)\cup \{p\}\right)=p$ and $p\in B_0\cap U$,
  there is $x_0\in B_0\cap U$ such that
  $f\left((X\setminus B_0)\cup \{x_0\}\right)=x_0$. For the same
  reason, taking $B_1\subset B_0\setminus\{x_0\}$, there is a point
  $x_1\in B_1\cap V$ with
  $f\left((X\setminus B_1)\cup \{x_1\}\right)=x_1$. Hence, by
  induction, there exists a sequence $\{B_n\}\subset \mathscr{B}$ and
  a sequence of points $\{x_n\}\subset X$ such that
  $B_{n+1}\subset B_n\setminus\{x_n\}$ and
  $$
    f\left((X\setminus B_{2n})\cup \{x_{2n}\}\right)=x_{2n}\in
    U\ \text{and}\ f\left((X\setminus B_{2n+1})\cup
      \{x_{2n+1}\}\right)=x_{2n+1}\in V.
    $$
    Since $T_n=(X\setminus B_n)\cup \{x_n\}$, $n<\omega$, is an
    increasing sequence of closed sets, it is
    $\tau_V$-convergent. Evidently,
    $\lim_{n\to\infty}x_n=\lim_{n\to\infty} f(T_n)\in
    \overline{U}\cap \overline{V}=\{p\}$.
\end{proof}

Let $\mathscr{F}_2(X)=\{S\subset X:1\leq |S|\leq 2\}$.  A selection
$\sigma$ for $\mathscr{F}_2(X)$ is called a \emph{weak selection} for
$X$. It generates a relation $\leq_\sigma$ on $X$ defined for
${x,y\in X}$ by $x\leq_\sigma y$ if $\sigma(\{x,y\})=x$
\cite[Definition 7.1]{michael:51}. This relation is both \emph{total}
and \emph{antisymmetric}, but not necessarily \emph{transitive}. We
write $x<_\sigma y$ whenever $x\leq_\sigma y$ and $x\neq y$, and
use the standard notation for the intervals generated by
$\leq_\sigma$. For instance, $(\leftarrow, p)_{\leq_\sigma}$ will
stand for all $x\in X$ with $x<_\sigma p$;
$(\leftarrow, p]_{\leq_\sigma}$ for that of all $x\in X$ with
$x\leq_\sigma p$; the intervals $(p,\to)_{\leq_\sigma}$,
$[p,\to)_{\leq_\sigma}$, etc., are defined in a similar way.\medskip

A weak selection $\sigma$ for $X$ is \emph{continuous} if it is
continuous with respect to the Vietoris topology on
$\mathscr{F}_2(X)$, equivalently if for every $p,q\in X$ with
$p<_\sigma q$, there are open sets $U,V\subset X$ such that $p\in U$,
$q\in V$ and $x<_\sigma y$ for every $x\in U$ and $y\in V$, see
\cite[Theorem 3.1]{gutev-nogura:01a}.  Thus, if $\sigma$ is continuous
and $p\in X$, then the intervals $(\gets,p)_{\leq_\sigma}$ and
$(p,\to)_{\leq_\sigma}$ are open in $X$ and $(\gets,p]_{\leq_\sigma}$
and $[p,\to)_{\leq_\sigma}$ are closed in $X$, see
\cite{michael:51}. However, the converse is not necessarily true
\cite[Example 3.6]{gutev-nogura:01a}, see also \cite[Corollary 4.2 and
Example 4.3]{gutev-nogura:09a}. The following property is actually
known, it will be found useful also in the rest of this paper.

\begin{proposition}
  \label{proposition-Sequences-v3:2}
  Let $X$ be a space which has a continuous weak selection $\sigma$
  and is totally disconnected at a point $p\in X$. If $\Delta_p$ is
  one of the intervals $(\gets,p)_{\leq_\sigma}$ or
  $(p,\to)_{\leq_\sigma}$, and $p$ is the limit of a sequence of
  points of $\Delta_p$, then $p$ is a countable intersection of clopen
  subsets of $\overline{\Delta_p}$. 
\end{proposition}

\begin{proof}
  According to \cite[Theorem
  4.1]{garcia-ferreira-gutev-nogura-sanchis-tomita:99}, see also
  \cite[Remark 3.5]{MR3478342}, $p$ is a $G_\delta$-point in
  $\overline{\Delta_p}$. Hence, since $X$ is totally disconnected
  at this point, it follows from \cite[Proposition
  5.6]{gutev-nogura:09a} that $p$ is a countable intersection of
  clopen subsets of $\overline{\Delta_p}$.  
\end{proof}

The remaining part of the proof of Theorem \ref{theorem-HSAP-v7:2} now
follows from the following observation, and Propositions
\ref{proposition-Sequences-v1:1} and \ref{proposition-Sequences-v1:2}.

\begin{proposition}
  \label{proposition-Sequences-v10:1}
  Let $X$ be a space which is totally disconnected at a point
  $p\in X$. If $X$ has a continuous weak selection and $p$ is the
  limit of a nontrivial convergent sequence, then $p$ is a cut point
  of $X$.
\end{proposition}

\begin{proof}
  Let $\sigma$ be a continuous weak selection for $X$. Then, by
  condition, $p$ is the limit of a sequence of points of
  $\Delta_p\subset X$, where $\Delta_p$ is one of the intervals
  $(\gets, p)_{\leq_\sigma}$ or $(p,\to)_{\leq_\sigma}$. Hence, by
  Proposition \ref{proposition-Sequences-v3:2}, there is a decreasing
  sequence $\{H_n\}$ of clopen subsets of $\overline{\Delta_p}$ and a
  sequence $\{x_n\}\subset \Delta_p$ convergent to $p$ such that
  $\bigcap_{n<\omega}H_n=\{p\}$ and $x_n\in H_n$, $n<\omega$. Taking
  subsequences if necessary, we can assume
  $x_n\in S_n=H_n\setminus H_{n+1}$ for all $n<\omega$. Then
  $U=\bigcup_{n<\omega}S_{2n}\subset \Delta_p\subset X\setminus \{p\}$
  is an open set with $\overline{U}=U\cup\{p\}$ because
  $\{x_{2n}\}\subset U$. Accordingly, for the set
  $V=X\setminus\overline{U}\subset X\setminus\{p\}$ we also have that
  $\overline{V}=V\cup\{p\}$ because $\{x_{2n+1}\}\subset V$. Thus, $p$
  is a cut point of $X$.
\end{proof}

\section{Countably-Approachable Points}
\label{sec:count-appr-points}

For a space $X$, the \emph{components} (called also \emph{connected
  components}) are the maximal connected subsets of $X$. They form a
closed partition $\mathscr{C}$ of $X$, and each element
$\mathscr{C}[x]\in \mathscr{C}$ containing a point $x\in X$ is called
the \emph{component} of this point.

\begin{proposition}
  \label{proposition-G-property-v1:1}
  Let $X$ be a space and $T,Z\in \mathscr{F}(X)$ be such that $Z$ is
  connected. If $f\in\sel[\mathscr{F}(X)]$ and $q=f(T\cup D)$ for some
  $D\in \mathscr{F}(Z)$, then $f(T\cup S)\in \mathscr{C}[q]$ for every
  $S\in \mathscr{F}(Z)$.
\end{proposition}

\begin{proof}
  Define a continuous map $f_T:\mathscr{F}(Z)\to X$ by
  $f_T(S)=f(T\cup S)$, for $S\in \mathscr{F}(Z)$. Then
  $Q=f_T\left(\mathscr{F}(Z)\right)$ is a connected subset of $X$
  because $\mathscr{F}(Z)$ is $\tau_V$-connected, see \cite[Theorem
  4.10]{michael:51}.  Accordingly, $Q\subset \mathscr{C}[q]$ because
  $q\in Q$.
\end{proof}

We now have the following relaxed condition for countably-approachable
points.

\begin{lemma}
  \label{lemma-Sequences-v3:1}
  Let $X$ be a space, $p=f(H)$ for some $f\in\sel[\mathscr{F}(X)]$ and
  ${H\in \mathscr{F}(X)}$, and $U\subset X\setminus\{p\}$ be an open
  set with $\overline{U}=U\cup\{p\}$. Also, let
  $\{H_n\}\subset\mathscr{F}(X)$ be a sequence which is
  $\tau_V$-convergent to $H$ such that for every $n<\omega$,
  \begin{equation}
    \label{eq:Sequences-v4:1}
    f(H_n)\in H_n\cap U\subset H_{n+1}\cap U\quad\text{and\quad
      $H_n\cap U$ is clopen.} 
  \end{equation}
Then $p$ is countably-approachable.
\end{lemma}

\begin{proof}
  In this proof, a crucial role will be played by the sets
  $L_n=H_n\cap U$ and $F_n=L_n\cup (H_{n+1}\setminus U)$, $n<\omega$.
  According to \eqref{eq:Sequences-v4:1},
  $f(H_n)\in L_n\subset L_{n+1}$ for every $n<\omega$. Since
  $\lim_{n\to\infty}f(H_n)=f(H)=p\notin U$, taking a subsequence if
  necessary, we can assume that
  \begin{equation}
    \label{eq:Sequences-v3:1}
    f(H_{n+1})\in L_{n+1}\setminus L_n=H_{n+1}\setminus F_n\quad
    \text{for every $n<\omega$.} 
  \end{equation}
  In these terms, let us also observe that 
  \begin{equation}
    \label{eq:HSAP-v31:1}
    \{F_n\}\subset \mathscr{F}(X)\ \ \text{is $\tau_V$-convergent to
      $H$.}
  \end{equation}
  Indeed, by \eqref{eq:Sequences-v4:1},
  $\{L_n\}\subset \mathscr{F}(X)$ is $\tau_V$-convergent to
  $L=\overline{\bigcup_{n<\omega}L_n}\subset H$. Take a finite open
  cover $\mathscr{V}$ of $H$ with $H\in \langle\mathscr{V}\rangle$,
  and set $\mathscr{V}_L=\{V\in \mathscr{V}:V\cap L\neq
  \emptyset\}$. Then there is $k<\omega$ such that
  $L_n\in \langle\mathscr{V}_L\rangle$ and
  $H_n\in \langle\mathscr{V}\rangle$ for every $n\geq k$. If $n\geq k$
  and $L_n\cap W=\emptyset$ for some $W\in\mathscr{V}$, then
  $W\notin \mathscr{V}_L$ and, therefore,
  $(H_{n+1}\setminus U)\cap W\neq \emptyset$. Accordingly,
  $F_n=L_n\cup (H_{n+1}\setminus U)\in
  \langle\mathscr{V}\rangle$. \smallskip

  Now, as in the proof of \cite[Lemma 4.4]{gutev:05a}, for every
  $n<\omega$ we will construct a closed set $T_n\subset X$ and a
  nonempty clopen set $S_n\subset L_{n+1}\setminus L_n$ such that
  \begin{equation}
    \label{eq:HSAP-v39:2}
    F_n\subset T_n \subset H_{n+1}\setminus S_n\quad \text{and}\quad
    f(T_n\cup\{x\})=x,\,\ \text{for every
      $x\in S_n$.}
\end{equation}
Briefly, $F_n\subset H_{n+1}$ and by \eqref{eq:Sequences-v4:1},
$H_{n+1}\setminus F_n=L_{n+1}\setminus L_n$ is clopen. Moreover, by
\eqref{eq:Sequences-v3:1}, $f(H_{n+1})\in H_{n+1}\setminus F_n$ and,
therefore, $q=f(F_n\cup E)\in E$ for some finite set
$E\subset H_{n+1}\setminus F_n$. Accordingly, we also have that
$\mathscr{C}[q]\subset H_{n+1}\setminus F_n$. Thus, setting
$D= E\cap \mathscr{C}[q]$, $K=E\setminus \mathscr{C}[q]$ and
$T_n=F_n\cup K$, it follows that
$D\subset \mathscr{C}[q]\subset H_{n+1}\setminus T_n$. Hence, by
Proposition~\ref{proposition-G-property-v1:1}, $f(T_n\cup\{y\})=y$ for
every $y\in \mathscr{C}[q]$ because $f(T_n\cup D)= f(F_n\cup
E)=q$. Finally, since $K$ is a finite set and $H_{n+1}\setminus F_n$
is clopen, $\mathscr{C}[q]\subset S$ for some clopen set
$S\subset H_{n+1}\setminus T_n$. Thus, the sets $T_n$ and
$S_n=\{x\in S: f(T_n\cup\{x\})=x\}$ are as required in
\eqref{eq:HSAP-v39:2}.\smallskip

To finish the proof, it only remains to show that
$\{S_n\}\subset \mathscr{F}(X)$ is $\tau_V$-convergent to $\{p\}$, see
\cite[Section 4]{gutev:05a}. So, take an open set $W$ containing $p$
and a finite family $\mathscr{V}$ of open sets such that
$H\in \langle\mathscr{V}\rangle$ and
$f(\langle\mathscr{V}\rangle)\subset W$. Then by condition and the
property in \eqref{eq:HSAP-v31:1}, there is $k<\omega$ with
$F_n,H_n\in \langle\mathscr{V}\rangle$ for every $n\geq k$.
Accordingly, for $n\geq k$ and $x\in S_n$, it follows from
\eqref{eq:HSAP-v39:2} that $x=f(T_n\cup\{x\})\in U$ because
$F_n\subset T_n\cup\{x\}\subset H_{n+1}$ implies that
$T_n\cup\{x\}\in \langle\mathscr{V}\rangle$. The proof is complete.
\end{proof}

\section{Approaching Trivial Components}
\label{sec:appr-totally-disc}

The \emph{quasi-component} $\mathscr{Q}[p]$ of a point $p\in X$ is the
intersection of all clopen subsets of $X$ containing this
point. Evidently, $\mathscr{C}[p]\subset \mathscr{Q}[p]$ for every
$p\in X$, but the converse is not necessarily true. However, these
components coincide for spaces with continuous weak selections, see
\cite[Theorem 4.1]{gutev-nogura:00b}. Hence, in this case, $X$ is
totally disconnected at $p\in X$ precisely when $\mathscr{C}[p]=\{p\}$
is trivial. \medskip

Here, we will prove the special case of Theorem
\ref{theorem-HSAP-v7:1} when the component of $X$ at $p\in X_\Theta$
is trivial. So, throughout this section, $f\in \sel[\mathscr{F}(X)]$
is a fixed selection such that $p=f(X)$ is a \emph{cut point} of $X$,
and $X$ is \emph{totally disconnected at $p$}. In this setting, the
trivial case is when $p$ is a $G_\delta$-point of $X$.

\begin{proposition}
  \label{proposition-Sequences-v10:2}
  If $p$ is a countable intersection of clopen subsets of $X$, then it
  is countably-approachable.
\end{proposition}

\begin{proof}
  Since $U=X\setminus\{p\}=\bigcup_{n<\omega}H_n$ for some increasing
  sequence $\{H_n\}\subset \mathscr{F}(X)$ of clopen sets, the
  property follows from Lemma \ref{lemma-Sequences-v3:1} by taking
  $H=X$.
\end{proof}

The rest of this section deals with the nontrivial case when $p$
is not a countable intersection of clopen sets. To this end, we shall
say that a pair $(U,V)$ of subsets of $X$ is a \emph{$p$-cut} of $X$
if $X\setminus\{p\}=U\cup V$ and
$\overline{U}\cap \overline{V}=\{p\}$.

\begin{proposition}
  \label{proposition-Sequences-v10:3}
  If $p$ is not a countable intersection of clopen subsets of $X$,
  then $X$ has a $p$-cut $(U,V)$ such that
  \begin{enumerate}[label=\upshape{(\roman*)}]
  \item\label{item:Sequences-v10:5} $p$ is a countable intersection of
    clopen subsets of\/ $\overline{U}$,
  \item\label{item:Sequences-v10:6} $V$ doesn't contain a sequence
    convergent to $p$.
  \end{enumerate}
\end{proposition}

\begin{proof}
  Since $p\in X_\Theta$ is a cut point, by Theorem
  \ref{theorem-HSAP-v7:2}, it is the limit of a sequence of points of
  $X\setminus\{p\}$. Therefore, $p$ is the limit of a sequence
  of points of $U\subset X$, where $U$ is one of the intervals
  $(\gets,p)_{\leq_f}$ or $(p,\to)_{\leq_f}$.  Hence, by
  Proposition \ref{proposition-Sequences-v3:2}, $p$ is a countable
  intersection of clopen subsets of $\overline{U}$. This implies that
  $V=X\setminus \overline{U}$ is not clopen in $X$ because $p$ is not
  a countable intersection of clopen subsets of $X$. For the same
  reason, $V$ doesn't contain a sequence convergent to
  $p$. Accordingly, this $p$-cut $(U,V)$ of $X$ is as required.
\end{proof}

The following two mutually exclusive cases finalise the proof of
Theorem \ref{theorem-HSAP-v7:1} when ${\mathscr{C}[p]=\{p\}}$. They
are based on two alternatives for the selection $f$ with respect
to the $p$-cut $(U,V)$ constructed in Proposition
\ref{proposition-Sequences-v10:3}, the set ${Y=\overline{V}}$ and a
fixed increasing sequence $\{T_n\}\subset \mathscr{F}(X)$ of clopen
sets with $\bigcup_{n<\omega}T_n=U$.

\begin{proposition}
  \label{proposition-Sequences-v5:3}
  Suppose that for every $S\in \mathscr{F}(Y)$ with $p\notin S$,
  \begin{equation}
    \label{eq:Sequences-v5:3}
    f(T_n\cup \{p\}\cup S)\neq p\quad \text{for all but finitely many
      $n<\omega$.} 
  \end{equation}
  Then $p$ is countably-approachable. 
\end{proposition}

\begin{proof}
  We proceed as in the proof of Proposition
  \ref{proposition-Sequences-v1:1}. Namely, take a local base
  $\mathscr{B}$ at $p$ in $Y$. Then for $B_0\in \mathscr{B}$ and
  $S_0=Y\setminus B_0$, by \eqref{eq:Sequences-v5:3}, there exists
  $n_0\geq 0$ such that $f(T_{n_0}\cup\{p\}\cup S_0)\neq p$. Next,
  using continuity of $f$, take $B_1\in \mathscr{B}$ such that
  $B_1\subset B_0$ and
  $f\left(T_{n_0}\cup K\cup S_0\right)\in T_{n_0}\cup S_0$ for every
  $K\in \Sigma(B_1)$, see \eqref{eq:HSAP-v21:4}. Hence, by Proposition
  \ref{proposition-G-property-v5:2}, we also have that
  $f\left(T_{n_0}\cup\overline{B_1}\cup S_0\right)\in T_{n_0}\cup
  S_0$. We can repeat the construction with $S_1=Y\setminus B_1$ and
  some $n_1>n_0$.  Thus, by induction, there exists a subsequence
  $\{T_{n_k}\}$ of $\{T_n\}$ and a decreasing sequence
  $\{B_k\}\subset \mathscr{B}$ such that if $S_k=Y\setminus B_k$,
  $k<\omega$, then
  \begin{equation}
    \label{eq:Sequences-v5:4}
    f\left(T_{n_k}\cup\overline{B_{k+1}}\cup S_k\right)\in T_{n_k}\cup S_k\quad
    \text{for every $k<\omega$.}
  \end{equation}
  By Proposition \ref{proposition-G-property-v3:1}, the sequence
  $\overline{B_{k+1}}\cup S_k$, $k<\omega$, is $\tau_V$-convergent to
  $Y$ because
  $\left\{B_k\setminus \overline{B_{k+1}}: k<\omega\right\}$ is
  pairwise disjoint and open in $Y$. Moreover,
  $\left\{T_{n_k}\right\}$ is $\tau_V$-convergent to $\overline{U}$
  being a subsequence of $\{T_n\}$. Hence,
  $H_k=T_{n_k}\cup\overline{B_{k+1}}\cup S_k$, $k<\omega$, is
  $\tau_V$-convergent to $X$. Accordingly,
  $p=f(X)=\lim_{k\to\infty} f(H_k)$. However, by
  \eqref{eq:Sequences-v5:4}, $f(H_k)\neq p$ for every
  $k<\omega$. Therefore, by \ref{item:Sequences-v10:6} of Proposition
  \ref{proposition-Sequences-v10:3}, $f(H_k)\in U$ for all but
  finitely many $k<\omega$. Thus, by Lemma \ref{lemma-Sequences-v3:1},
  the point $p$ is countably-approachable.
\end{proof}

\begin{proposition}
  \label{proposition-Sequences-v5:2}
  Suppose that there exists $S\in \mathscr{F}(Y)$ with $p\notin S$,
  and a subsequence $\left\{T_{n_j}\right\}$ of $\{T_n\}$ such that
  \begin{equation}
    \label{eq:Sequences-v5:2}
    f\left(T_{n_j}\cup\{p\}\cup S\right)=p\quad \text{for all 
      $j<\omega$.}
  \end{equation}
 Then $p$ is countably-approachable. 
\end{proposition}

\begin{proof}
  Evidently, we can assume that \eqref{eq:Sequences-v5:2} holds for
  all $n<\omega$.  Next, using Theorem \ref{theorem-HSAP-v7:2} and
  \ref{item:Sequences-v10:6} of Proposition
  \ref{proposition-Sequences-v10:3}, take a sequence
  $\{x_n\}\subset U$ which is convergent to $p$ and $x_n\in T_n$ for
  every $n<\omega$. Since $f$ is continuous and the sequence
  $T_0\cup\{x_n\}\cup S$, $n<\omega$, is $\tau_V$-convergent to
  $T_0\cup \{p\}\cup S$, it follows from \eqref{eq:Sequences-v5:2}
  that $f(T_0\cup\{x_{n_0}\}\cup S)=x_{n_0}$ for some $n_0<\omega$. We
  can repeat this with $T_{n_0}$. Namely, the sequence
  $T_{n_0}\cup \{x_n\}\cup S$, $n>n_0$, is $\tau_V$-convergent to
  $T_{n_0}\cup\{p\}\cup S$. Hence, for the same reason,
  $f(T_{n_0}\cup\{x_{n_1}\}\cup S)=x_{n_1}$ for some $n_1>n_0$. Thus,
  by induction, there are subsequences $\{x_{n_k}\}$ of $\{x_n\}$ and
  $\left\{T_{n_k}\right\}$ of $\{T_n\}$ such that
  $f(T_{n_k}\cup\{x_{n_{k+1}}\}\cup S)=x_{n_{k+1}}$ for every
  $k<\omega$. Then $H_k= T_{n_k}\cup S$, $k<\omega$, is a
  $\tau_V$-convergent sequence with $\lim_{k\to\infty}f(H_k)=p$,
  because $H_k\subset H_k\cup\{x_{n_{k+1}}\} \subset H_{k+1}$ for
  every $k<\omega$. Furthermore, by \ref{item:Sequences-v10:6} of
  Proposition \ref{proposition-Sequences-v10:3}, $f(H_k)\in U$ for all
  but finitely many $k<\omega$.  Therefore, just like before, Lemma
  \ref{lemma-Sequences-v3:1} implies that $p$ is
  countably-approachable.
\end{proof}

\section{Approaching Nontrivial Components} 
\label{sec:appr-nontr-comp}

Here, we will finalise the proof of Theorem \ref{theorem-HSAP-v7:1}
with the remaining case when $X$ is not totally disconnected at
$p$. To this end, let us recall that a space $X$ is
\emph{weakly orderable} if there exists a coarser orderable topology
on $X$ with respect to some linear order on it (called
\emph{compatible} for $X$). The weakly orderable spaces were
introduced by Eilenberg \cite{eilenberg:41}, and are often called
``Eilenberg orderable''.\medskip

Each connected space $Z$ with a continuous weak selection $\sigma$ is
weakly orderable with respect to $\leq_\sigma$, see \cite[Lemmas
7.2]{michael:51}. The following simple observation was implicitly
present in the proof of \cite[Theorem 1.5]{gutev-nogura:00b}. In this
observation, and what follows, $\noncut(Z)$ are the noncut points of a
connected space $Z$, and $\cut(Z)$ --- the cut points of $Z$.

\begin{proposition}
  \label{proposition-G-property-v5:1}
  Let $X$ be a space and $p=f(X)$ for some
  $f\in\sel[\mathscr{F}(X)]$. Then $p\in \noncut(\mathscr{C}[p])$.
\end{proposition}

\begin{proof}
  Set $Z=\mathscr{C}[p]$ and assume that $p\in H=\cut(Z)$. Since $H$
  is open in $X$ (see \cite[Corollary 2.7]{gutev:07a}) and $f$ is
  continuous, $f(\langle \mathscr{U}\rangle)\subset H\subset Z$ for
  some finite open cover $\mathscr{U}$ of $X$. Take a finite set
  $T\subset X\setminus Z$ with
  $Y=T\cup Z\in\langle \mathscr{U}\rangle$. Then it follows from
  Proposition \ref{proposition-G-property-v1:1} that
  $g(S)=f(T\cup S)\in S$ for every $S\in \mathscr{F}(Z)$. Accordingly,
  $g:\mathscr{F}(Z)\to Z$ is a continuous selection with
  $g(Z)\in H=\cut(Z)$. However, this is impossible because $Z$ is
  weakly orderable with respect to $\leq_g$ and $g(Z)$ is the first
  $\leq_g$-element of $Z$, see \cite[Lemmas 7.2 and 7.3]{michael:51}
  and \cite[Corollary 2.7]{gutev:07a}.
\end{proof}

In the rest of this section, $p\in X_\Theta$ is a cut point such that
the component $\mathscr{C}[p]$ is not a singleton. In this case, by
Proposition \ref{proposition-G-property-v5:1}, $p$ is a noncut point
of $\mathscr{C}[p]$. Thus, we can also fix a $p$-cut $(U,V)$ of $X$
such that $\mathscr{C}[p]\subset \overline{V}$. Accordingly,
$Y=\overline{U}$ is totally disconnected at $p$. In this setting, the
remaining part of the proof of Theorem \ref{theorem-HSAP-v7:1}
consists of showing that $p$ is countably-approachable in $Y$. To this
end, we will first show that $\overline{V}$ can itself be assumed to
be connected.

\begin{proposition}
  \label{proposition-HSAP-v23:1}
  Let $f:\mathscr{F}(X)\to X$ be a continuous selection with
  $f(X)=p$. Then there exists a nondegenerate connected subset
  $Z\subset \mathscr{C}[p]$ such that $p\in Z$ and $X_*=Y\cup Z$ has a
  continuous selection $f_*:\mathscr{F}(X_*)\to X_*$ with
  $f_*(X_*)=p$.
\end{proposition}

\begin{proof}
  Since $H=\mathscr{C}[p]$ has a continuous weak selection and
  $p\in \noncut(H)$, the space $H$ is weakly orderable with respect to
  a linear order $\leq$ such that $p\leq x$ for every $x\in H$, see
  \cite[Lemma 7.2]{michael:51} and \cite[Corollary
  2.7]{gutev:07a}. Accordingly, each closed interval
  $Z_x=[p,x]_\leq\in \mathscr{F}(H)$, $x\in \cut(H)$, is a connected
  subset of $H$, see \cite[Theorem 1.3]{MR0339099}. Moreover, if
  $T=\overline{V\setminus H}$, then $f(Y\cup H\cup T)=f(X)=p\in H$.
  Thus, by Propositions \ref{proposition-G-property-v1:1} and
  \ref{proposition-G-property-v5:1},
  \begin{equation}
    \label{eq:G-property-v8:1}
    f(Y\cup Z_x\cup T)\in \noncut(Z_x)\cup T=\{p,x\}\cup T,\quad
    x\in\cut(H).
  \end{equation}
  Evidently, the resulting family
  $\mathscr{S}=\left\{Y\cup Z_x\cup T: x\in \cut(H)\right\}$ is
  up-directed. Therefore, by Pro\-position
  \ref{proposition-G-property-v5:2}, it is $\tau_V$-convergent to
  $\overline{\bigcup\mathscr{S}}=X$. Hence, by
  \eqref{eq:G-property-v8:1}, $f(Y\cup Z_q\cup T)=p$ for some
  $q\in \cut(H)$ because $\lim_{S\in \mathscr{S}}
  f(S)=f(X)=p$. Finally, let $Z=Z_q$, $X_*=Y\cup Z$ and
  $\mathscr{T}=\left\{S\in \mathscr{F}(X_*): f(S\cup T)\in T\right\}$.
  Then $\mathscr{T}$ is a $\tau_V$-clopen set in $\mathscr{F}(X_*)$
  because $T$ is clopen in $X_*\cup T$. So, we may define a continuous
  selection $f_*:\mathscr{F}(X_*)\to X_*$ by letting for
  $S\in\mathscr{F}(X_*)$ that
  \[
    f_*(S)=
    \begin{cases}
      f(S) &\text{if $S\in \mathscr{T}$,\quad and} \\
      f(S\cup T) &\text{if $S\notin \mathscr{T}$.}
    \end{cases}
  \]
  Since $f(X_*\cup T)=f(Y\cup Z\cup T)=f(Y\cup Z_q\cup T)=p\notin T$,
  we get that $X_*\notin \mathscr{S}$. Accordingly, we also have that
  $f_*(X_*)=f(X_*\cup T)=p$.
\end{proof}

Since the space $X_*=Y\cup Z$ in Proposition
\ref{proposition-HSAP-v23:1} has all properties of $X$ relevant to our
case, we can identify $X$ with this space. In this refined setting,
the fixed $p$-cut $(U,V)$ of $X$ has the extra property that
$Z=\overline{V}$ is connected, while $Y=\overline{U}$ is the same
as before.

\begin{proposition}
  \label{proposition-Sequences-v10:5}
  If $f:\mathscr{F}(X)\to X$ is a continuous selection with
  $f(X)=p$, then $f(Y)=p$. Moreover, 
  \begin{equation}
    \label{eq:G-property-v5:1}
    f(Y\cup \{q\})=p\quad \text{for every $q\in Z$.}
  \end{equation}
\end{proposition}

\begin{proof}
  Since $Z$ is connected, it follows from Proposition
  \ref{proposition-G-property-v1:1} that $f(Y\cup S)\in Z$ for every
  $S\in \mathscr{F}(Z)$. Accordingly,
  $f(Y)=f(Y\cup\{p\})=p$. Regarding \eqref{eq:G-property-v5:1}, we
  argue by contradiction. Namely, assume that
  \begin{equation}
    \label{eq:HSAP-v23:1}
    f(Y\cup\{q\})=q\quad \text{for
      some $q\in Z$ with $q\neq p$.}   
  \end{equation}

  Next, as in the proof of Proposition \ref{proposition-HSAP-v23:1},
  using that $Z$ is weakly orderable and $p\in \noncut(Z)$, take a
  compatible linear order $\leq$ on $Z$ such that $p\leq z$ for every
  $z\in Z$. Then by \eqref{eq:HSAP-v23:1}, $p<q$ and we now have that
  \begin{equation}
    \label{eq:G-property-v1:2}
    f(Y\cup S)\in S,\quad \text{whenever $S\in
      \mathscr{F}\left([z,\to)_\leq\right)$ for some $z>p$.}  
  \end{equation}
  Briefly, for $z>p$ and $S\in \mathscr{F}\left([z,\to)_\leq\right)$,
  either $S\subset [q,\to)_\leq$ or $q\in [z,\to)_\leq$. Since all
  $\leq$-intervals of $Z$ are connected, \eqref{eq:G-property-v1:2}
  follows from \eqref{eq:HSAP-v23:1} and Proposition
  \ref{proposition-G-property-v1:1}.\smallskip

  This now implies that the continuous map $g(T)=f(Y\cup T)$,
  $T\in \mathscr{F}(Z)$, is a selection for $\mathscr{F}(Z)$. Indeed,
  for $T\in \mathscr{F}(Z)$ with $T\neq\{p\}$, set
  \[
    \mathscr{S}=\big\{T\cap[z,\to)_{\leq}: z\in T\setminus\{p\}\big\}.
    \]
    Since the family $\mathscr{S}$ is up-directed in $\mathscr{F}(Z)$,
    by Proposition \ref{proposition-G-property-v5:2}, it is
    $\tau_V$-convergent to $\overline{\bigcup\mathscr{S}}$. Moreover,
    by \eqref{eq:G-property-v1:2}, $g(S)=f(Y\cup S)\in S$ for every
    $S\in \mathscr{S}$.  Therefore,
    $g(T)\in \overline{\bigcup\mathscr{S}}\subset T$ because
    $Y\cup \left(\overline{\bigcup\mathscr{S}}\right)=Y\cup
    T$. However, $\mathscr{F}(Z)$ has at most two continuous
    selections --- taking the minimal element, or taking the maximal
    element of each $T\in \mathscr{F}(Z)$ \cite[Lemma
    7.3]{michael:51}. Accordingly, $g(T)=\min_\leq T$ for every
    $T\in \mathscr{F}(Z)$ because
    $g(Z)=f(Y\cup Z)=f(X)=p=\min_\leq Z$. But this is impossible
    because by \eqref{eq:HSAP-v23:1},
    $q=g(\{q\})=f(Y\cup \{q\})=\min_\leq\{p,q\}=p$.
\end{proof}

The following final observation completes the proof of Theorem
\ref{theorem-HSAP-v7:1}.  

\begin{proposition}
  \label{proposition-Sequences-v10:4}
  If $f:\mathscr{F}(X)\to X$ is a continuous selection with $f(X)=p$,
  then $Y\setminus\{p\}$ contains a sequence convergent to $p$. In
  particular, $p$ is a cut point of $Y$, and is therefore also
  countably-approachable. 
\end{proposition}

\begin{proof}
  Let $\mathscr{B}$ be a local base at $p$ in $Y$. Then there
  exists $B_0\in \mathscr{B}$ such that
  \begin{equation}
    \label{eq:G-property-v5:2}
    f\left((Y\setminus B)\cup\{p\}\right)\in Y\setminus B,\quad
    \text{for every 
      $B\in \mathscr{B}$ with $B\subset B_0$.}
  \end{equation}
  Indeed, assume that \eqref{eq:G-property-v5:2} fails, and let
  $\mathscr{B}_*$ be the collection of all $B\in \mathscr{B}$ such
  that $f\left((Y\setminus B)\cup\{p\}\right)=p$. Then $\mathscr{B}_*$
  is also a local base at $p\in Y$. Hence, the family
  $\mathscr{S}=\{Y\setminus B: B\in \mathscr{B}_*\}$ is an up-directed
  cover of $Y\setminus\{p\}$ and by
  Proposition~\ref{proposition-G-property-v5:2}, it is
  $\tau_V$-convergent to $Y$. Moreover, by assumption,
  $f(S\cup\{p\})=p$ for every $S\in\mathscr{S}$. Therefore, by
  Proposition \ref{proposition-G-property-v1:1}, we also have that
  $f(S\cup\{q\})=q$ for every $q\in Z$ and $S\in
  \mathscr{S}$. Accordingly,
  $f(Y\cup\{q\})=\lim_{S\in \mathscr{S}}f(S\cup\{q\})=q$ for every
  $q\in Z$. However, by \eqref{eq:G-property-v5:1} of Proposition
  \ref{proposition-Sequences-v10:5}, this is impossible.\smallskip

  Thus, \eqref{eq:G-property-v5:2} holds and by Propositions
  \ref{proposition-Sequences-v1:1} and
  \ref{proposition-Sequences-v10:5}, there exists a sequence of points
  of $Y\setminus\{p\}$ which is convergent to $p$. According to
  Proposition \ref{proposition-Sequences-v10:1}, this implies that $p$
  is a cut point of $Y$. Therefore, by Theorem \ref{theorem-HSAP-v7:2}
  (applied with $X=Y$), $p$ is also countably-approachable in
  $Y$. Since $Y=\overline{U}=U\cup\{p\}$ and
  $U\subset X\setminus\{p\}$ is open, $p$ is countably-approachable in
  $X$ as well.
\end{proof}

\section{Point-Maximal Selections} 
\label{sec:point-maxim-select}

Recall that a point $p\in X$ in an arbitrary space $X$ is
\emph{noncut} if it is not a cut point.  The prototype of such points
can be traced back to Michael's nowhere cuts defined in
\cite{MR167968}. In his terminology, a subset $A\subset X$
\emph{nowhere cuts} $X$ \cite{MR167968} if $A$ has an empty interior
(i.e.\ $A$ is \emph{thin}) and whenever $p\in A$ and $U$ is a
neighbourhood of $p$ in $X$, then $U\setminus A$ does not split into
two disjoint open sets both having $p$ in their closure. Evidently,
the singleton $\{p\}$ nowhere cuts $X$ for each noncut point $p\in
X$. A slight variation of this concept was consider in \cite{MR819737}
(under the name `does not cut') and in \cite{MR1745451} (under the
name `nowhere disconnects').\medskip

As commented in the Introduction, the equality
$X_\Theta^*\cap X_\Omega=X_\Omega^*$ in \eqref{eq:HSAP-v21:3} of
Theorem \ref{theorem-HSAP-v7:3} is known, see \eqref{eq:HSAP-v20:2}
and \eqref{eq:HSAP-v21:1}. Here, we will prove the following refined
version of this theorem showing that the members of
$X_\Theta\setminus X_\Omega$ possess a similar property with respect
to the connected components. To this end, let us recall that a
(closed) subset $C\subset X$ has a \emph{clopen base} if for each
neighbourhood $U$ of $C$ there exists a clopen set $H\subset X$ with
$C\subset H\subset U$. In case $C=\{p\}$ is a singleton, we simply say
that $X$ is \emph{zero-dimensional} at $p\in X$.

\begin{theorem}
  \label{theorem-Sequences-v13:1}
  Let $X$ be a space with $\sel[\mathscr{F}(X)]\neq \emptyset$ and
  $p\in X_\Theta\setminus X_\Omega$. Then $\mathscr{C}[p]$ has a
  clopen base and $p\in X_\Theta^*$.
\end{theorem}

Evidently, the essential case in Theorem \ref{theorem-Sequences-v13:1}
is when $\mathscr{C}[p]$ is not a clopen set, otherwise the property
follow easily from known result and Proposition
\ref{proposition-G-property-v1:1}. Thus, in the rest of this section,
$\mathscr{C}[p]$ will be assumed to be not clopen.\medskip

The next lemma covers the case of $\mathscr{C}[p]=\{p\}$ in Theorem
\ref{theorem-Sequences-v13:1}. 

\begin{lemma}
  \label{lemma-HSAP-v23:1}
  Let $X$ be a space with $\sel[\mathscr{F}(X)]\neq \emptyset$. If $X$
  is totally disconnected at some point
  $p\in X_\Theta\setminus X_\Omega$, then $X$ is zero-dimensional at
  $p$ and $p\in X_\Theta^*$.
\end{lemma}

The proof of this lemma is base on the following two simple
observations.

\begin{proposition}
  \label{proposition-Sequences-v11:1}
  Let $f:\mathscr{F}(X)\to X$ be a continuous selection, $p\in X$ with
  $\mathscr{C}[p]=\{p\}$, and $K\in \mathscr{F}(X)$ be such that
  $p\notin K$ and $f(K\cup S)=p$ for every closed set $S\subset X$
  with $p\in S$. Then $X$ has a clopen base at $p$.
\end{proposition}

\begin{proof}
  We follow the idea in the proof of \cite[Theorem
  1.4]{gutev-nogura:00b}, see also
  \cite{bertacchi-costantini:98}. Take an open set $U\subset X$ with
  $p\in U\subset X\setminus K$, and set $F=X\setminus U$. Since
  $f(F)\neq p$, there exists a clopen set $T\subset X$ with
  $f(F)\in T$ and $p\notin T$. Then $f^{-1}(T)$ is a $\tau_V$-clopen
  subset $\mathscr{F}(X)$. Take a maximal chain
  $\mathscr{M}\subset f^{-1}(T)$ with $F\in\mathscr{M}$. Then
  $M=\overline{\bigcup\mathscr{M}}$ is the maximal element of
  $\mathscr{M}$, and therefore $M$ is clopen in $X$ because
  $f^{-1}(T)$ is $\tau_V$-clopen. Moreover, $K\subset F\subset M$
  because $F\in \mathscr{M}$. Finally, $M$ doesn't contain $p$ because
  $f(M)\neq p$. Indeed, if $p\in M$, then by condition
  $f(M)=f(K\cup M)=p$, which is impossible. Thus, $H=X\setminus M$ is
  a clopen set with $p\in H\subset U$.
\end{proof}

\begin{proposition}
  \label{proposition-Sequences-v12:1}
  Let $f:\mathscr{F}(X)\to X$ be a continuous selection, $p\in X$ with
  $\mathscr{C}[p]=\{p\}$, and $K\in \mathscr{F}(X)$ be such that
  $p\notin K$ and $f(K\cup S)=p$ for every closed set $S\subset X$
  with $p\in S$. Then $p\in X_\Theta^*$.
\end{proposition}

\begin{proof}
  Since $f(K\cup\{p\})=p\notin K$, there exists a finite family
  $\mathscr{V}$ of open subsets of $X$ such that
  $K\cup\{p\}\in \langle\mathscr{V}\rangle$ and
  $f(\langle\mathscr{V}\rangle)\subset X\setminus K$. Then by
  Proposition \ref{proposition-Sequences-v11:1}, there exists a clopen
  set $H$ such that $p\in H\subset X\setminus K$ and
  $H\subset \bigcap\{V\in \mathscr{V}: p\in V\}$. Accordingly,
  $f(K\cup S)\in S$ for every $S\in \mathscr{F}(H)$. We can now define
  a continuous selection $h:\mathscr{F}(X)\to X$ by letting for
  $S\in \mathscr{F}(X)$ that
  \[
    h(S)=
    \begin{cases}
      f(S) &\text{if $S\cap H=\emptyset$,\ \ and} \\
      f(K\cup S_H)& \text{if $S_H=S\cap H\neq \emptyset$.}
    \end{cases}
  \]
  This shows that $p$ is selection maximal. Indeed,
  $p\in S\in \mathscr{F}(X)$ implies that $p\in S_H=S\cap H$ and by
  the property of $K$, we have that $h(S)=f(K\cup S_H)= p$.
\end{proof}

\begin{proof}[Proof of Lemma \ref{lemma-HSAP-v23:1}]
  According to Propositions \ref{proposition-Sequences-v11:1} and
  \ref{proposition-Sequences-v12:1}, it suffices to show that there
  exists $K\in \mathscr{F}(X)$ such that
  \begin{equation}
    \label{eq:Sequences-v13:1}
    p\notin K\quad \text{and}\quad \text{$f(K\cup S)=p$,\ \ for every
      $S\in \mathscr{F}(X)$ with $p\in S$.}
  \end{equation}

  To this end, let $\mathscr{O}$ be the collection of all open subsets
  containing $p$, and $\mathscr{B}\subset \mathscr{O}$ be that one of
  those $B\in \mathscr{O}$ for which
  $f((X\setminus B)\cup \{p\})\neq p$. If $\mathscr{B}$ is a local
  base at $p$, then by Proposition \ref{proposition-Sequences-v1:1},
  $X\setminus\{p\}$ contains a sequence convergent to $p$. Hence, by
  Proposition \ref{proposition-Sequences-v10:1}, $p$ must be a cut
  point of $X$. However, by Theorem \ref{theorem-HSAP-v7:1}, this is
  impossible because $p\notin X_\Omega$. Accordingly, there exists
  $U\in \mathscr{O}$ such that $K=X\setminus U\neq \emptyset$ and
  $\mathscr{V}=\{V\in \mathscr{O}:V\subset U\}$ doesn't contain any
  member of $\mathscr{B}$, namely  $f((X\setminus V)\cup\{p\})=p$ for
  every $V\in \mathscr{V}$. To see that this $K$ is as required, take
  a closed set $S\subset X$ with $p\in S$, and set
  $\mathscr{L}=\{S\setminus V: V\in \mathscr{V}\}$. Then
  $X\setminus (K\cup L)\in \mathscr{V}$, $L\in \mathscr{L}$, and
  therefore $f(K\cup L\cup \{p\})=p$ for every $L\in
  \mathscr{L}$. Moreover, by Proposition
  \ref{proposition-G-property-v5:2},
  $\mathscr{H}=\{K\cup L\cup \{p\}: L\in \mathscr{L}\}$ is
  $\tau_V$-convergent to
  $\overline{\bigcup\mathscr{H}}=\overline{K\cup S}=K\cup S$ being an
  up-directed cover of $K\cup S$. Since $f$ is continuous,
  this implies that $f(K\cup S)=p$ which shows
  \eqref{eq:Sequences-v13:1}.
\end{proof}

The other case of Theorem \ref{theorem-Sequences-v13:1} is covered
by the following lemma. 

\begin{lemma}
  \label{lemma-HSAP-v23:2}
  Let $X$ be a space with $\sel[\mathscr{F}(X)]\neq \emptyset$, and
  $p\in X_\Theta\setminus X_\Omega$ be such that
  $\mathscr{C}[p]\neq\{p\}$.  Then $\mathscr{C}[p]$ has a clopen base
  and $p\in X_\Theta^*$.
\end{lemma}

In this lemma, according to Theorem \ref{theorem-HSAP-v7:1} (see also
Proposition \ref{proposition-G-property-v5:1}), $p$ is both a noncut
point of $X$ and a noncut point of $\mathscr{C}[p]$. Since
$\mathscr{C}[p]$ is not clopen in $X$, it has another noncut point
$q\in \mathscr{C}[p]$ defined by the property that
$q\in \overline{X\setminus \mathscr{C}[p]}$. In particular, $q$ is a
cut point of $X$. Thus, in this case, $U=X\setminus \mathscr{C}[p]$
and $V=\mathscr{C}[p]\setminus\{q\}$ form a $q$-cut of $X$ such that
$Y=\overline{U}$ is totally disconnected at $q$ and
$Z=\overline{V}=\mathscr{C}[p]$. In this setting, we have the
following properties of $Y$ and $Z$. 

\begin{proposition}
  \label{proposition-Sequences-v13:1}
  There exists a nonempty finite set $K\subset U$ such that for every
  closed set $S\subset Y$, the map $f_{(K,S)}(T)=f(K\cup S\cup T)$,
  $T\in \mathscr{F}(Z)$, is a continuous $p$-maximal selection for
  $\mathscr{F}(Z)$.
\end{proposition}

\begin{proof}
  Since $f(X)=p\in V$, there exists a finite open cover $\mathscr{W}$
  of $X$ with $X\in \langle\mathscr{W}\rangle$ and
  $f(\langle\mathscr{W}\rangle)\subset V$. Take a finite set
  $K\subset U$ such that $K\cap W\neq \emptyset$ for every
  $W\in \mathscr{W}$ with $W\cap Y\neq\emptyset$. Then $K$ has the
  property that
  \begin{equation}
    \label{eq:Sequences-v14:1}
    f(K\cup S\cup Z)=p,\quad \text{for every closed set $S\subset Y$.}
  \end{equation}
  Indeed, in this case, $f(K\cup S\cup Z)\in V\subset \mathscr{C}[p]$
  because $K\cup S\cup Z\in \langle\mathscr{W}\rangle$. Hence, by
  Proposition \ref{proposition-G-property-v5:1}, $f(K\cup S\cup Z)=p$
  because $q\notin V$.\smallskip

  For a closed subset $S\subset Y$, the map $f_{(K,S)}$ is continuous
  and by Proposition~\ref{proposition-G-property-v1:1},
  $f_{(K,S)}(T)= f(K\cup S\cup T)\in Z$ for every
  $T\in \mathscr{F}(Z)$. Hence, by \eqref{eq:Sequences-v14:1} and
  \cite[Lemmas 7.2 and 7.3]{michael:51}, it only suffices to show that
  $f_{(K,S)}$ is a selection for $\mathscr{F}(Z)$. If
  $T\in \mathscr{F}(Z)$ and $q\notin S$, then $f_{(K,S)}(T)\in T$
  because $K\cup S\subset U\subset X\setminus Z$. Otherwise, if
  $q\in S$, we set $F=S\setminus\{q\}$ and distinguish the following
  two cases:\smallskip

  \noindent \textbf{(i)} If $F$ is a closed set, as remarked above,
  $f_{(K,F)}$ is a selection for $\mathscr{F}(Z)$.  Therefore, by
  \eqref{eq:Sequences-v14:1}, $f_{(K,F)}$ is `$q$-minimal' in the
  sense that $f_{(K,F)}(T)= q$ precisely when $T= \{q\}$ because
  $q\in\noncut(\mathscr{C}[p])$, see \cite[Lemma 7.3]{michael:51} and
  \cite[Corollary 2.7]{gutev:07a}. In other words,
  $f_{(K,S)}(T)=f_{(K,F)}(T\cup \{q\})\in T$ for every
  $T\in \mathscr{F}(Z)$. \smallskip

  \noindent \textbf{(ii)} If $F$ is not closed, by the previous
  case, $f_{(K,E)}(T)=f(K\cup E\cup T)\in T$ for every
  $E\in \Sigma(F)$, see \eqref{eq:HSAP-v21:4}. Moreover, by
  Proposition \ref{proposition-G-property-v5:2},
  $\mathscr{H}=\left\{K\cup E: E\in\Sigma(F)\right\}$ is an
  up-directed family which is $\tau_V$-convergent to $K\cup
  S$. Accordingly,
  \[
    f_{(K,S)}(T)=f(K\cup S\cup T) =\lim_{H\in \mathscr{H}}f\left(H\cup
      T\right)\in \overline{T}=T. \qedhere
  \]
\end{proof}

\begin{proof}[Proof of Lemma \ref{lemma-HSAP-v23:2}]
  According to Proposition \ref{proposition-Sequences-v13:1}, there
  exists a nonempty finite set $K\subset U=X\setminus \mathscr{C}[p]$
  such that $f(K\cup S\cup \{q\})=q$ for every closed set
  $S\subset Y=\overline{U}$. Since $Y$ is totally disconnected at $q$,
  it follows from Proposition~\ref{proposition-Sequences-v11:1} that
  $q$ has a clopen base in $Y$. This implies that $\mathscr{C}[p]$ has
  a clopen base in $X$. To show the remaining part of this lemma, as
  in the proof of Proposition \ref{proposition-Sequences-v12:1}, take
  a clopen set $H\subset Y$ such that $q\in H\subset Y\setminus K$ and
  $f(K\cup S)\in S$ for every $S\in \mathscr{F}(H)$. Then $L=H\cup Z$
  is a clopen subset of $X$ with the same property. Indeed, take any
  $S\in \mathscr{F}(L)$. If $S\subset Y$, then $S\subset H$ and
  therefore $f(K\cup S)\in S$. If $S\setminus Y\neq \emptyset$, set
  $D=S\cap Y$ and $T=S\cap Z$. Then by
  Proposition~\ref{proposition-Sequences-v13:1},
  $f(K\cup S)=f(K\cup D\cup T)\in T\subset S$. Hence, just like
  before, we can define a continuous $p$-maximal selection
  $h:\mathscr{F}(X)\to X$ by
  \[
    h(S)=
    \begin{cases}
      f(S) &\text{if $S\cap G=\emptyset$,\ \ and} \\
      f(K\cup S_L) &\text{if $S_L=S\cap L\neq \emptyset$.}
    \end{cases}
  \]
  Indeed, if $p\in S\in \mathscr{F}(X)$, then $p\in S_L=S\cap
  L$. Moreover, $D_L=S_L\cap Y$ is closed in $Y$ and
  $p\in T_L=S_L\cap Z\in \mathscr{F}(Z)$.  According to Proposition
  \ref{proposition-Sequences-v13:1},
  $h(S)=f(K\cup S_L)=f(K\cup D_L\cup T_L)=p$.
\end{proof}

%%%\bibliographystyle{amsplain-ab}
%%%\bibliography{gutev}
\end{document}